\documentclass[11pt,letterpaper]{article}
\usepackage{amsfonts, amsmath, amssymb, amscd, amsthm, graphicx}
\usepackage{epsfig, color}

\makeatletter
\def\blfootnote{\xdef\@thefnmark{}\@footnotetext}
\makeatother

\newtheorem{thm}{Theorem}[section]
\newtheorem{cor}[thm]{Corollary}
\newtheorem{lem}[thm]{Lemma}

\newtheorem{prob}[thm]{Problem}

\theoremstyle{definition}

\theoremstyle{remark}
\newtheorem{rem}[thm]{Remark}

\newfont{\eufm}{eufm10}

\newcommand{\e}{\varepsilon }

\renewcommand{\phi}{{\rm Lab }}
\renewcommand{\kappa}{\varkappa}

\renewcommand{\ll}{\left\langle\hspace{-.7mm}\left\langle }
\newcommand{\rr}{\right\rangle\hspace{-.7mm}\right\rangle }

\begin{document}

\title{$L^2$-Betti numbers and non-unitarizable groups without free subgroups}
\author{D. V. Osin \thanks{This work has been supported by the NSF grant DMS-0605093}}
\date{}%

\maketitle

\begin{abstract}
We show that there exist non-unitarizable groups without non-abelian free subgroups.
Both torsion and torsion free examples are constructed. As a by-product, we show that there exist finitely generated torsion groups
with non-vanishing first $L^2$-Betti numbers. We also relate the well-known problem of whether every hyperbolic group is residually finite to an open question about approximation of $L^2$-Betti numbers.
\end{abstract}

\section{Introduction}
Let $G$ be a group, $H$ a Hilbert space. Recall that a
representation $\pi\colon\;G\to B(H)$ is {\it unitarizable} if
there exists an invertible operator ${S}\colon\;H\to H$ such that $g\to
{S}^{-1} \pi(g){S}$ is a unitary representation of $G$. A (locally compact)
group $G$ is {\it unitarizable} if every uniformly bounded
representation $\pi\colon\;G\to B(H)$ is unitarizable.

In 1950, Day \cite{Day} and Dixmier~\cite{Dix} proved that every amenable
group is unitarizable. The question of whether the converse holds
has been open since then. A good survey of the current research in this direction is given in \cite{Pis}.

The simplest examples of a non-unitarizable groups are non-abelian free groups. (An explicit construction of a uniformly bounded non-unitarizable representation can be found \cite{MZ}). Since unitarizability passes to subgroups, every group containing a non-abelian free subgroup is not unitarizable as well. However, the answer to the following question was
unknown until now.

\begin{prob}
Does there exist a non-unitarizable group without non-abelian free
subgroup?
\end{prob}

Note that if there is such a group, it is a non-amenable group without non-abelian free subgroups. The existence of such groups remained a fundamental open problem for many years until the first examples were constructed by Olshanskii in \cite{Ols80}. The aim of this note is to answer the question affirmatively. Namely we prove the following.

\begin{thm}\label{1}
There exists a finitely generated torsion non-unitarizable group.
\end{thm}

The proof of Theorem \ref{1} is a combination of a sufficient condition for non-unitarizability found by Epstein and Monod \cite{EM}, a recent result of Peterson and Thom \cite{PT} about first $L^2$-Betti numbers of groups defined by periodic relations, and some older techniques related to hyperbolic groups \cite{Gro, Ols93}. Though the property of being torsion is crucial for this approach, we show that the torsion group from Theorem \ref{1} can be used to construct examples of completely different nature.

\begin{thm}\label{2}
There exists a finitely generated torsion free non-unitarizable group without free subgroups.
\end{thm}

As a by-product, we also obtain some new results about $L^2$-Betti
numbers of groups. Recall that if $G$ is a torsion free group
satisfying the Atiyah Conjecture (or even a weaker property
($\ast$) introduced in \cite{PT}), then $\beta_1^{(2)}(G)>0$
implies the existence of non-abelian free subgroups in $G$
\cite[Theorem 4.1]{PT}. For finitely presented residually
$p$-finite groups, where $p$ is a prime, even a stronger result
holds. Namely $\beta_1^{(2)}(G)>0$ for such a group $G$ implies
that $G$ is large \cite{Lack}. These examples lead to a natural
question of whether non-vanishing of the first $L^2$-Betti number
always implies the existence of non-abelian free subgroups. The
following theorem shows that the answer is negative.

\begin{thm}\label{torb}
There exists a finitely generated torsion group with non-vanishing first $L^2$-Betti number.
\end{thm}

Moreover, Theorem \ref{torb} allows us to relate a question about
approximation of $L^2$-Betti numbers to one of the most intriguing
open problems about hyperbolic groups. Recall that a group $G$ is
{\it residually finite} if for every element $g\ne 1$ of $G$ there
is a homomorphism $\e\colon G\to Q$, where $Q$ is finite, such
that $\e(g)\ne 1$. By the Approximation Theorem of L\"uck
\cite{Luck}, for every residually finite finitely presented group
$G$ and every nested sequence of finite index normal subgroups $\{
N_i\} $ of $G$ with trivial intersection, one has
\begin{equation}\label{Luck}
\beta_1^{(2)} (G) =\lim\limits_{i\to \infty}
\frac{b_1(N_i)}{[G:N_i]} ,
\end{equation}
where $b_1(N_i)$ is the ordinary first Betti number of $N_i$. The
following question is still open.
\begin{prob}\label{lueck}
Does the approximation hold for any finitely generated residually finite group?
\end{prob}

The other problem is well-known. For a survey of the theory of
hyperbolic groups we refer to \cite{Gro,BH}.
\begin{prob}\label{hyprf}
Is every hyperbolic group residually finite?
\end{prob}

We show that if every hyperbolic group is residually finite, then the group from Theorem \ref{torb} can be made residually finite as well. However this contradicts (\ref{Luck}) since $\beta _1(N_i)=0$ for any subgroup of a torsion group. Thus we obtain the following.

\begin{cor}
At least one of the two problems has a negative solution.
\end{cor}

\noindent {\bf Acknowledgment} I am grateful to Nicolas Monod for drawing my attention to the paper \cite{EM} and stimulating discussions. I am also grateful to Jesse Peterson for explaining results of \cite{PT}.

\section{Torsion groups with positive first $L^2$-betty numbers}

Recall that a group is {\it elementary} if it contains a cyclic
subgroup of finite index. For every hyperbolic group $G$ and every
element $g\in G$ of infinite order, there exists a (unique)
maximal elementary subgroup $E(g)\le G$ containing $g$ (see, e.g.,
\cite[Lemma 1.16]{Ols93}.

Given an element $g$ of a group $G$, we denote by $\ll g\rr$ the normal closure of $g$ in $G$, i.e., the smallest normal subgroup of $G$ containing $g$. Our main tool in this section is the following result about adding higher powered relations to hyperbolic groups. Up to little changes it was conjectured by Gromov \cite{Gro}. It can easily be extracted from the proof of (a more complicated) Theorem 3 in \cite{Ols93}. Since the result we need  does not formally follow from \cite[Theorem 3]{Ols93}, we briefly explain how to derive it from other results of \cite{Ols93} for convenience of the reader.

\begin{lem}[Olshanskii, \cite{Ols93}]\label{perrel}
Let $G$ be a hyperbolic group, $S$ a finite subset of $G$, $E$ a maximal elementary subgroup of $G$, $C$ a finite
index normal cyclic subgroup of $E$. Suppose that $C=\langle x\rangle $. Then for every sufficiently large integer $n$ the following conditions hold.
\begin{enumerate}
\item[(1)] The quotient group $G/\ll x^n\rr $ is hyperbolic.

\item[(2)] The image of the element $x$ in $G/\ll x^n\rr $ has order $n$.

\item[(3)] The natural homomorphism $G\to G/\ll x^n\rr $ is injective on $S$.
\end{enumerate}
\end{lem}

\begin{proof}
Let $X$ be a finite generating set of $G$, $W$ a shortest word in $X\cup X^{-1}$ representing $x$ in $G$. By \cite[Lemma 4.1]{Ols93} the set of all cyclic shifts of the words $W^{\pm m}$ satisfies a small cancellation condition, which implies properties (1)-(3) for $G=\langle G\mid W^n=1\rangle \cong G/\ll x^n\rr $ by \cite[Lemma 6.7]{Ols93} if $m=m(G,x,S)$ is sufficiently large.
\end{proof}

For a background on $L^2$-Betti numberst we refer the reader to \cite{Luck-book}. In what follows we assume that $1/|G| =0$ if a group $G$ has infinite order. Recall that for $G=G_1\ast \cdots \ast G_n$, we have
\begin{equation}\label{bfp}
\beta_1^{(2)} (G)=n-1+ \sum_{i=1}^n \left(\beta_1^{(2)} (G_i)-\frac{1}{|G_i|}\right)
\end{equation}
(see\cite{Luck08}). The following theorem of Peterson and Thom \cite{PT} will allow us to control first $L^2$-Betti number after adding higher-powered relations to $G$.

\begin{thm} \cite[Theorem 3.2]{PT} \label{PT} Let $G$ be an infinite
countable discrete group. Assume that there exist subgroups $G_1,\dots,G_n$ of $G$,
such that
$$G = \langle G_1, \dots , G_n \mid  r_1^{w_1},\dots,r_k^{w_k} \rangle,$$
for some elements $r_1,\dots,r_k \in G_1 \ast \cdots \ast G_n$ and positive integers $w_1,\dots,w_k$.
Suppose in addition that the order of $r_i$ in $G$ is $w_i$.
Then, the following inequality holds:
$$\beta_{1}^{(2)}(G) \geq n-1 + \sum_{i=1}^n \left( \beta_1^{(2)}(G_i) - \frac{1}{|G_i|}\right) - \sum_{j=1}^k \frac{1}{w_j}. $$
\end{thm}

\begin{thm}\label{tor}
For every positive integer $n$ and every $\e>0$, there exists an $n$-generated torsion group $T$ with $\beta _1^{(2)}(T) \ge n-1-\e $. Moreover, if every hyperbolic group is residually finite, then the group $T$ can be additionally made residually finite.
\end{thm}

\begin{proof}
Roughly speaking, the main idea of the proof is to start with the free product $G=\mathbb Z_m \ast\cdots\ast \mathbb Z_m$ and then add periodic relations $r^w=1$ (one by one) for all $r\in  G$, where $m$ and all $w=w(r)$ are large enough. We are going to use Theorem \ref{PT} to prove that the first $L^2$-betti number of the groups obtained at each step is close to the number of free factors in $G$. A continuity argument will then help us carry over this estimate to the limit group. The only difficulty is to verify the hypothesis of Theorem \ref{PT} concerning orders of $r_i$'s and to ensure that the images of $\mathbb Z_m$'s remain isomorphic to $\mathbb Z_m$ on each step. This is done by using hyperbolic groups and Lemma \ref{perrel}.

Let $m$ be an integer such that $n/m <\e $. Let $G=G_1\ast \cdots \ast G_n$, where $G_i\cong \mathbb Z_m$ for each $i=1, \ldots, n$. We enumerate all elements of $G=\{ 1=g_0, g_1, g_2, \ldots \} $, and construct the group $T$ by induction. Let $T_0=G$. Suppose that a group $$T_i=\langle G_1, \ldots , G_n\mid r_1^{w_1}, \ldots , r_{k_i}^{w_{k_i}}\rangle, $$ ${k_i}\le i$, has already been constructed for some $i\ge 0$. In what follows, we use the same notation for elements of $T_0$ and their images in $T_i$. We assume that:
\begin{enumerate}
\item[(a)] $T_i$ is hyperbolic.

\item[(b)] The natural maps $G_l\to T_i$ are injective for $l=1, \ldots, n$. In particular, we may think of $G_l$'s as subgroups of $T_i$.

\item[(c)] $|r_j|=w_j$ in $T_i$ for $j=1, \ldots , {k_i}$.

\item[(d)] $\sum_{j=1}^{k_i} \frac{1}{w_j} +\frac{n}{m} <\e .$

\item[(e)] Elements $g_0, \ldots , g_i$ have finite orders in $T_i$.
\end{enumerate}
Observe that the inductive assumption trivially holds for $T_0$. By Theorem \ref{PT}, conditions (b), (c), and (d) imply that
$$
\beta^{(2)}_1 (T_i)>n-1-\e .
$$

The group $T_{i+1}$ is obtained from $T_i$ in the following way. If the image of $g_{i+1}$ has finite order in $T_i$, we set $k_{i+1}=k_i$ and $T_{i+1}=T_i$. Otherwise let $C$ be an (infinite) finite index cyclic normal subgroup of $E(g_{i+1})$. Since $|E(g_{i+1})/C|$ is finite, there exists  $m>0$ such that $g_{i+1}^m\in C$. Consequently, $\langle g_{i+1}^{m}\rangle $ is normal in $E(g_{i+1})$. Passing to $\langle g_{i+1}^m\rangle $, we may assume that $C=\langle g_{i+1}^m\rangle $ without loss of generality. Let $k_{i+1}=k_i+1$, $r_{k_{i+1}}=g_{i+1}$ and $$S=\left( \bigcup\limits_{l=1}^n G_l\right) \cup \left( \bigcup\limits_{j=1}^{k_i} \langle r_j\rangle \right) .$$  Applying Lemma  \ref{perrel} to the group $G=T_i$, the element $x=g_{i+1}^m$, and the subset $S$, we obtain that for every large enough integer $s$, the quotient group
\begin{equation}\label{Ti+1}
T_{i+1}=\langle T_i \mid r_{k_{i+1}}^{sm}\rangle =\langle G_1, \ldots , G_n\mid r_1^{w_1}, \ldots , r_{k_i}^{w_{k_i}}, r_{k_{i+1}}^{sm}\rangle
\end{equation}
is hyperbolic, $|g_{i+1}^m|=s$ (hence $|r_{k_{i+1}}|=|g_{i+1}|=ms$) in $T_{i+1}$, and the natural homomorphism $T_i\to T_{i+1}$ is injective on $S$. The later condition ensures that $|r_j|=w_j$ in $T_{i+1}$ for $j=1, \ldots , {k_i}$ and the natural maps $G_l\to T_{i+1}$ remain injective for $l=1, \ldots, n$ . By (d), we may choose $s$ such that
$$
\sum_{j=1}^{{k_i}} \frac{1}{w_j} +\frac{1}{ms} +\frac{n}{m} <\e .
$$
Letting  $k_{i+1}=k_i+1$ and  $w_{k_{i+1}}=ms$ completes the inductive step.

Let $$T=\langle G_1, \ldots , G_n\mid r_1^{w_1}, r_{2}^{w_{2}}, \ldots \rangle $$ be the inductive limit of the groups $T_i$ and the natural homomorphisms $T_i\to T_{i+1}$.  By (e) the image of every element $g_i$ has finite order in $T$, i.e., $T$ is a torsion group.

Note that the sequence $\{T_i\}$ converges to $T$ in the topology of marked group presentations. (For details about this topology we refer to \cite{Pich}.) Indeed this is always true whenever we have a sequence of normal subgroups $N_1\le N_2 \le \ldots $ of a group $T_0$, $T=T_0/\bigcup\limits_{i=1}^\infty N_i$, and $T_i=T_0/N_i$. (In our case $N_i$ is the normal closure of $r_1, \ldots , r_{k_i}$ in $T_0$). By semi-continuity of the first $L^2$-Betti number (see \cite{Pich}), we obtain $$\beta^{(2)}_1 (T)\ge \lim\limits_{i\to \infty}   \beta^{(2)}_1 (T_i) \ge n-1-\e .$$

Suppose now that every hyperbolic group is residually finite. Then
we adjust our construction as follows. Let $X$ be a finite
generating set of $T_0$ and let $d_i$ denote the word metric on
$T_i$ corresponding to the natural image of $X$ in $T_i$. For
every $i\in \mathbb N$, we choose a homomorphism $\tau_i\colon
T_i\to Q_i$, where $Q_i$ is finite, such that $\tau _i(t)\ne 1$
whenever $d_i(t,1)\le i$ and $t\ne 1$. Such a homomorphism always
exists as $T_i$ is hyperbolic and hence it is residually finite by
our assumption.

Note that passing from $T_i$ to $T_{i+1}$ according to (\ref{Ti+1}) we may always choose $s$ to be a multiple of any given non-zero integer. Thus we may assume that $w_j$ is divisible by $|Q_i|$ for any  $i,j\in \mathbb N$, $j>i$. This implies that for every $i\in \mathbb N$, the kernel of the natural homomorphism $T_i\to T$ is contained in ${\rm Ker} (\tau _i)$ and hence $\tau _i$ factors through $T_i\to T$. Let $\sigma _i$ be the corresponding homomorphism $T\to Q_i$.

Denote by $d$ the word metric on $T$ with respect to the natural image of the set $X$. If $s$ is a nontrivial element of $T$ such that $d(s, 1)=i$, then there is an element $t\in T_i$ such that $d_i (t, 1)\le  i$ and $s$ is the natural image of $t$ in $T$. According to our construction, $\tau _i(t)\ne 1$ and hence $\sigma _i(t)\ne 1$. Thus the group $T$ is residually finite.
\end{proof}

\section{Non-unitarizable groups without free subgroups}

Given a finitely generated group $H$, we denote by ${\rm rk \,}(H)$ its rank.

\begin{thm}[Epstein-Monod \cite{EM}]\label{EM}
Let $G$ be a unitarizable group. Then the ratio $\beta_1^{(2)}(H)/\sqrt{ {\rm rk\, } (H)}$ is uniformly bounded on the set of all finitely generated subgroups of $G$.
\end{thm}

We are now ready to prove Theorem \ref{1}.

\begin{proof}[Proof of Theorem \ref{1}]
By Theorem \ref{tor} for every integer $n\ge 2$, there exists a group $G_n$ generated by $n$ elements such that $\beta_1^{(2)} (G_n)\ge n-2$. Let $G$ be the direct product of the family $\{G_n\mid n\ge 2\}$. Clearly $G$ is a torsion group and is not unitarizable by Theorem \ref{EM}. To complete the proof it remains to recall that every countable torsion group embeds into a torsion group generated by $2$ elements \cite{Ols91}, and every group containing a non-unitarizable subgroup is non-unitarizable itself.
\end{proof}

To construct torsion free examples, we need another result about hyperbolic groups. Similarly to Lemma \ref{perrel}, it can be easily extracted from the proof of Theorem 2 in \cite{Ols93}. We make this extraction for convenience of the reader and refer to \cite{Ols93} for details and terminology.

\begin{lem}[Olshanskii, \cite{Ols93}]\label{glue}
Let $G$ be a torsion free hyperbolic group, $H$ a non-elementary subgroup of $G$, $t_1, \ldots , t_m$ elements of $G$. Then there exist elements $r_1, \ldots , r_m\in H$ such that the following conditions hold for the quotient group $G_1=G/\ll r_1t_1, \ldots , r_mt_m\rr $.
\begin{enumerate}
\item[(1)] $G_1$ is torsion free hyperbolic.

\item[(2)] The natural image of $H$ is a non-elementary subgroup of $G_1$.
\end{enumerate}
Observe that the images of the elements $t_1, \ldots , t_m$ belong to the image of $H$ in $G_1$.
\end{lem}

\begin{proof}
Since $G$ is torsion free, all elementary subgroups of $G$ are cyclic. Let $g$ be any non-trivial element of $H$ such that $E(g)=\langle g\rangle $. Let $l$ be a positive integer, $x_1, \ldots, x_l\in H$ elements provided by \cite[Lemma 3.7]{Ols93}. Let $W, X_0, X_1, \ldots , X_l$ be shortest words in a finite set of generators of $G$ representing $g, r_1, x_1, \ldots, x_l$ respectively.

Using \cite[Lemma 4.2]{Ols93} and triviality of finite subgroups in $G$, we obtain that the set of all cyclic shifts of the words $(X_0W^mX_1W^m\cdots X_lW^m)^{\pm 1}$ satisfies a small cancellation condition,which implies   (1) and (2) for the group $$G_1=\langle G \mid X_0W^mX_1W^m\cdots X_lW^m =1\rangle$$ by \cite[Lemma 6.7]{Ols93} if $m$ and $l$ are large enough. Note that $G_1\cong G/\ll r_1t_1\rr$, where $t_1$ is the element of $H$ represented by $W^mX_1W^m\cdots X_lW^m$. Doing the same procedure for $r_2, \ldots , r_m$ we prove the lemma by induction.
\end{proof}

\begin{proof}[Proof of Theorem \ref{2}]
We are going to construct the desired group $G$ as a (torsion free) extention $1\to H\to G\to T\to 1$, where $T$ is a non-unitarizable torsion group provided by Theorem \ref{1} and $H$ has no non-abelian free subgroups. It is easy to show that every such a group $G$ is not unitarizable and does not contain non-abelian free subgroups.

More precisely, let
\begin{equation}\label{T}
T= \langle x,y \mid r_1, r_2, \ldots \rangle ,
\end{equation}
be a presentation of a non-unitarizable torsion group. Without loss of generality we may assume $T$ to be generated by $2$ elements (see the proof of Theorem \ref{1}).  Again we proceed by induction. Let $G_0=\langle x,y, a,b \rangle $ be the free group with basis $x,y, a,b$. In what follows we construct a series of quotients of $G_0$. As in the proof of Theorem \ref{tor}, we keep the same notation for elements of $G_0$ and their images in these quotient groups.

Clearly $H_0=\langle a, b\rangle $ is a non-elementary subgroup of $G_0$. Hence by Lemma \ref{glue}, there exist elements $u_1, \ldots , u_8, v_1\in H$ such that the quotient group
$$
G_1=\langle G_0\mid a^xu_1, a^{x^{-1}}u_2, b^xu_3, b^{x^{-1}}u_4, a^yu_5, a^{y^{-1}}u_6, b^yu_7, b^{y^{-1}}u_8, r_1v_1\rangle
$$
is torsion free hyperbolic, and the image $H_1$ of $H$ in $G_1$ is non-elementary. Without loss of generality we may assume that $u1, \ldots , u_8$ and $v_1$ are words in $\{ a^{\pm 1}, b^{\pm 1}\} $.

We enumerate all finitely generated subgroups $H_1=R_1, R_2 ,\ldots $ of $H_1$. Suppose that for some $i\ge 1$, we have already constructed a group
$$
G_i=\left\langle a,b,x,y\; \left| \; \begin{array}{c}
                                    a^xu_1, a^{x^{-1}}u_2, b^xu_3, b^{x^{-1}}u_4, a^yu_5, a^{y^{-1}}u_6, b^yu_7, b^{y^{-1}}u_8 \\
                                    r_1 v_1, \ldots , r_iv_i\\
                                     w_1, \ldots , w_{k_i}
                                  \end{array}
\right. \right\rangle
$$
such that the following conditions hold. By $H_i$ we denote the subgroup of $G_i$ generated by $a$ and $b$ (i.e., the image of $H_0$ in $G_i$).
\begin{enumerate}
\item[(a)] The group $G_i$ is torsion free hyperbolic.

\item[(b)] The subgroup $H_i=\langle a,b\rangle $ of $G_i$ is non-elementary.

\item[(c)] $v_1, \ldots , v_i$ and $w_1, \ldots , w_{k_i}$ are words in $\{ a^{\pm 1}, b^{\pm 1}\} $.

\item[(d)] For every $j=1, \ldots , i$, the image of $R_j$ in $G_i$ is either cyclic or coincides with the image of $H_i$.
\end{enumerate}
Clearly these conditions hold for $i=1$. Relations $ w_1, w_2, \ldots , w_{k_i}$ are absent in this case.

The group $G_{i+1}$ is obtained from $G_i$ in two steps. First, by parts (a), (b) of the inductive assumption and  Lemma \ref{glue}, we may choose a word $v_{i+1}$ in  $\{ a^{\pm 1}, b^{\pm 1}\} $ such that the quotient group $K_i=G_i/\ll r_{i+1}v_{i+1} \rr $ is torsion free hyperbolic, and the natural image of $H_i$ in $K_i$ is non-elementary.

Further if the natural image of $R_{i+1}$ in $K_i$ is cyclic, we set $G_i=K_i$ and $k_{i+1}=k_i$. Otherwise the image of $R_{i+1}$ in $K_i$ is non-elementary. Indeed it is well-known and easy to prove that every torsion free elementary group is cyclic. Thus we can apply Lemma \ref{glue} to the image of $R_{i+1}$ in $K_i$ and elements $a,b$. Let $z_1, z_2$ be elements of the image of $R_{i+1}$ in $K_i$ such that the quotient group $G_{i+1}=K_i/\ll az_1, bz_2\rr $ is torsion free hyperbolic and the image of $R_{i+1}$ in $G_{i+1}$ is non-elementary.
Recall that $R_{i+1}\le H_1=\langle a, b\rangle $. Hence we can assume that $z_1, z_2$ are words in $\{ a^{\pm 1}, b^{\pm 1}\}$. Since $a=z_1^{-1}$ and $b=z_2^{-1}$ in $G_{i+1}$, the image of $R_{i+1}$ in $G_{i+1}$ coincides with the subgroup $H_{i+1}=\langle a,b\rangle$ of $G_{i+1}$. In particular, $H_{i+1}$ is non-elementary. Note that $az_1$, $bz_2$ are words in $\{ a^{\pm 1}, b^{\pm 1}\}$ as well. Let $w_{k_i+1}=az_1$, $w_{k_i+2}=bz_2$. The inductive step is completed.

Let now $G$ be the inductive limit of the sequence $G_0, G_1, \ldots $. That is,
\begin{equation}\label{Ginfty}
G_i=\left\langle a,b,x,y\; \left| \; \begin{array}{c}
                                    a^xu_1, a^{x^{-1}}u_2, b^xu_3, b^{x^{-1}}u_4, a^yu_5, a^{y^{-1}}u_6, b^yu_7, b^{y^{-1}}u_8 \\
                                    r_1 v_1, r_2v_2, \ldots \\
                                     w_1, w_2, \ldots
                                  \end{array}
\right. \right\rangle
\end{equation}
Let also $H=\langle a,b\rangle $ be the natural image of $H_1$ in $G$. Note that the elements $a^x, a^{x^{-1}}, b^x, b^{x^{-1}}, a^y, a^{y^{-1}}, b^y, b^{y^{-1}} $ belong to $H$ in $G$, as $u_1, \ldots , u_8$ are words in $\{ a^{\pm 1}, b^{\pm 1}\}$. Hence $H$ is a normal subgroup of $G$. Obviously $G/H\cong T$. Indeed after imposing additional relations $a=1$ and $b=1$, the relations corresponding to the first and the third rows of (\ref{Ginfty}) disappear and the second row of (\ref{Ginfty}) becomes $r_1, r_2, \ldots $ (see (c)). After removing $a, b$ from the set of generators, we obtain exactly the presentation (\ref{T}) of $T$.

Thus the group $G$ splits as $1\to H\to G\to T\to 1$. Note that every finitely generated proper subgroup of $H$ is cyclic. Indeed, let $Q$ be a finitely generated subgroup of $H$, $P$ some finitely generated preimage of $Q$ in $H_1$. Then $P=R_i$ for some $i$. The natural homomorphism $P\to Q$ obviously factors through the image of $P=R_i$ in $G_i$. However the image of $R_i$ in $G_i$ is either cyclic or coincides with $H_i$ by (d). Therefore, we obtain that $Q$ is either cyclic or coincides with $H$.

Let $F$ be a nontrivial finitely generated free subgroup of $G$. Then $F\cap H$ is cyclic. Note also that $F\cap H$ is normal in $F$ and $F\cap H\ne 1$ since $F/(F\cap H)\cong FH/H\le T$ is a torsion group. Thus $F$ has a nontrivial normal cyclic subgroup, and hence $F$ is cyclic itself. This shows that $G$ contains no non-abelian free subgroups. Further suppose that some element $g\ne 1$ has finite order in $G$. This means that for some $n>0$, the relation $g^n=1$ follows from relations of the presentation (\ref{Ginfty}). Hence it follows from some finite set of relations of (\ref{Ginfty}), i.e., $g^n=1$ holds in $G_i$ for some $i$ contrary to (a). Finally we note that $G$ is not unitarizable since it surjects onto a non-unitarizable group $T$.
\end{proof}

\begin{rem}
 Let $G$ be finitely generated group, $X$ a finite generating set of $G$, $\varkappa (G, l^2(G), X)$ the Kazhdan
constant of the left regular representation $\lambda _G$ with respect to $X$. More precisely, $\varkappa (G, l^2(G),
X)$ is defined as the supremum of all $\e>0$, such that for every vector $u\in l^2(G)$ of norm $1$, there exists $x\in
X$ such that $\| \lambda _G(x) u -u\| \ge \e $. Recall that a finitely generated group $G$ is {\it amenable} if and only is $\varkappa (G, l^2(G), X)=0$ for every finite generating set $X$ of $G$ \cite{Hul}.

Let
\begin{equation}\label{wa}
\alpha (G)=\inf\limits_{\langle X\rangle =G} \varkappa (G,l^2(G), X)=0,
\end{equation}
where the infimum is taken over all finite generating sets  $X$ of $G$. A finitely generated group $G$ is called {\it weakly amenable} if $\alpha (G)=0$. (A similar notion of weak amenability was considered in \cite{Aetall}.) Clearly every amenable group is weakly amenable. The converse was shown to be wrong in \cite{Osi02}.

Combining methods of \cite{Osi02} and \cite{Osi04}, one can construct a (finitely generated) non-unitarizable torsion group $T$ without free subgroup such that every non-elementary hyperbolic group surjects onto $T$. In particular, every such a group is weakly amenable \cite{Osi02}. This shows that the absence of free subgroups does not imply unitarizability even being combined with weak amenability.

\end{rem}

\end{document}